\documentclass[oneside]{amsart}
\usepackage{amsmath,amssymb,amsfonts,amsthm}
\usepackage{hyperref}
\usepackage{tikz}

\theoremstyle{plain}
\newtheorem{thm}{Theorem}[section]

\newtheorem{prop}[thm]{Proposition}

\newtheorem*{definition}{Definition}

\DeclareMathOperator{\crank}{crank}

\begin{document}

\title{Refining Blecher and Knopfmacher's Integer Partition Fixed Points}

\author{Brian Hopkins}\address{Saint Peter's University, Jersey City NJ, USA}\email{bhopkins@saintpeters.edu}

\maketitle

\begin{abstract}
Recently, Blecher and Knopfmacher explored the notion of fixed points in integer partitions.  Here, we distinguish partitions with a fixed point by which value is fixed and analyze the resulting triangle of integers.  In particular, we confirm various identities for diagonal sums, row sums, and antidiagonal sums (which are finite for this triangle) and establish a four-term recurrence for triangle entries analogous to Pascal's lemma for the triangle of binomial coefficients.  The partition statistics crank and mex arise.  All proofs are combinatorial.
\end{abstract}

\section{Introduction and statement of primary results}

Given a positive integer $n$, a partition of $n$ is a collection of positive integers $\lambda = (\lambda_1, \ldots, \lambda_j)$ with $\sum \lambda_i = n$.  The $\lambda_i$, called parts, are ordered so that $\lambda_1 \ge \cdots \ge \lambda_j$.  We use the notation $\lambda \vdash n$ and $|\lambda| = n$ to indicate that $\lambda$ is a partition of $n$.  Write $P(n)$ for the set of partitions of $n$ and $p(n) = \#P(n)$ for the number of partitions of $n$.  Recall the convention that $p(0) = 1$.

Fixed points are a pervasive concept in mathematics, including the study of permutations.  Recently, Blecher and Knopfmacher brought the idea of fixed points to integer partitions \cite[Section 4]{bk}.

\begin{definition} \label{fp}
A partition $\lambda$ has a fixed point if there is an index $i$ for which $\lambda_i = i$.
\end{definition}

For example, among the partitions of 5, $\alpha=(3,2)$ has a fixed point since $\alpha_2 = 2$ while $\beta = (3,1,1)$ has no fixed points.  With the convention of writing parts in nonincreasing order, it is clear that a partition has at most one fixed point.

Here we consider a refinement of the fixed point.
\begin{definition} \label{fp}
Given $n \ge 1$, let $F(n,d)$ be the partitions of $n$ with fixed point $d$ and $f(n,d) = \#F(n,d)$ the number of these partitions.
\end{definition} 

For example, $F(5,1) = \{(1,1,1,1,1)\}$ (which we sometimes write as $1^5$) and $F(5,2) = \{(3,2), (2,2,1)\}$ account for all the partitions of 5 with a fixed point, so that $f(5,1) = 1$, $f(5,2) = 2$, and $f(5,k) = 0$ for all $k \ge 3$.

The irregular triangle of integers $f(n,d)$, shown in Figure \ref{fig1}, has an internal recurrence similar to Pascal's lemma for the triangle of binomial coefficients and several sum identities.  In particular, we establish the following four results.

\begin{figure}
\begin{tabular}{r|rrrr}
$n \backslash d$ & \; 1 & 2 & 3 & 4 \\ \hline
1 & 1 & & &  \\
2 & 1 & & &  \\
3 & 1 & & &   \\
4 & 1 & 1 & &  \\
5 & 1 & 2 & &  \\
6 & 1 & 4 & &  \\
7 & 1 & 6 & &  \\
8 & 1 & 9 & &  \\
9 & 1 & 12 & 1 & \\
10 & 1 & 16 & 2 &  \\
11 & 1 & 20 & 5 &  \\
12 & 1 & 25 & 9 &  \\
13 & 1 & 30 & 16 &  \\
14 & 1 & 36 & 25 &  \\
15 & 1 & 42 & 39 &  \\
16 & 1 & 49 & 56 & 1  \\
17 & 1 & 56 & 80 & 2 \\
18 & 1 & 64 & 109 & 5 \\
19 & 1 & 72 & 147 & 10 \\
20 & 1 & 81 & 192 & 19
\end{tabular}
\caption{Triangle of $f(n,d)$ values for $n \le 20$.} \label{fig1}
\end{figure}

In Proposition \ref{colrecur} below, we will show that column $d$ of the triangle satisfies a degree $d^2$ recurrence.  However, using terms of column $d-1$ allows the following recurrence which is linear in $d$.

\begin{thm} \label{fixedPas}
For $d \ge 2$ and $n \ge d^2$,
\[f(n,d) = f(n-d+1,d) + f(n-d,d) - f(n-2d + 1, d) + f(n-2d+1,d-1).\]
\end{thm}

There are several sums in the $f(n,d)$ triangle that connect to other partitions and partition statistics.  Next, we recall the relevant definitions.

A partition $\lambda$ can be represented as a Ferrers diagram where row $i$ consists of $\lambda_i$ dots.  The conjugate $\lambda'$ switches the rows and columns of the Ferrers diagram.  The largest square that can be placed in the Ferrers diagram of $\lambda$ is called its Durfee square (see Figures \ref{fig2} and \ref{fig3} below).  Symbolically, the dimension of the Durfee is the greatest index $i$ for which $\lambda_i \ge i$.

The crank statistic was named by Dyson in 1944 who predicted its use in a combinatorial understanding of Ramanujan's modulo 11 congruence \cite{d}.  The ``elusive crank'' was defined by Andrews and Garvan in 1988 \cite{ag} and has become a very important statistic for integer partitions.  To define it for a partition $\lambda$, let $\omega(\lambda)$ be the number of parts 1 in $\lambda$ and $\mu(\lambda)$ the number of parts of $\lambda$ greater than $\omega(\lambda)$.  Then
\[\text{crank}(\lambda) =\begin{cases} \lambda_1 & \text{ if $\omega(\lambda)=0$},\\
\mu(\lambda) - \omega(\lambda) & \text{ if $\omega(\lambda)>0$}.
\end{cases}\]
Write $M(m,n)$ for the number of partitions of $n$ with crank $m$.

The mex (minimal excludant) of a partition $\lambda$ is the smallest positive integer which is not a part of $\lambda$.  See \cite{an,hs20} for connections between the crank and mex statistics.

Related to the diagonal sum result, let $a(n)$ be the number of partitions $\lambda$ of $n$ such that the size of the Durfee square of $\lambda$ is not a part of $\lambda$ \cite[A118199]{o}.

\begin{thm} \label{diag}
For $n \ge 1$, 
\[f(n,1) + f(n-1,2) + \cdots = a(n+1).\]
\end{thm}

The $n$th row sum automatically gives the total number of partitions of $n$ with a fixed point.  By recent work of the author and James Sellers \cite{hs24}, this connects to the crank and mex as follows.

\begin{thm} \label{row}
For $n \ge 1$,
\[f(n,1) + f(n,2) + \cdots = \# \{ \lambda \vdash n \mid \crank{\lambda} \equiv 0 \bmod 2 \} = \sum_{m \ge 1} M(m,n).\]
\end{thm}

Because, for $n \ge 2$, the $n$th row of the triangle has fewer than $n$ nonzero entries, antidiagonal sums are finite.

\begin{thm} \label{antidiag}
For $n \ge 1$,
\[f(n,1) + f(n+1, 2) + \cdots = p(n-1).\]
\end{thm}

In the next section, we address the columns of the $f(n,d)$ triangle and prove Theorem \ref{fixedPas}.  In Section 3, we prove the sum results.  All proofs are combinatorial.

\section{Recurrences in the refined fixed point triangle}

We first consider each column of the triangle, that is, the sequence of $f(n,d)$ values for a fixed $d$ specifying the position of the fixed point.  The next proposition gives the generating function for this sequence.  The statement of the result uses the Pochhammer symbol: For $n\geq 1$, \[(a;q)_n = (1-a)(1-aq)\cdots(1-aq^{n-1}).\] 

\begin{prop} \label{colrecur}
For a fixed $d \ge 1$, the sequence $f(n,d)$ satisfies
\[ \sum_{n \ge 1} f(n,d) \, q^n = \frac{q^{d^2}}{(q;q)_{d-1} (q;q)_d}.\] 
It follows that $f(n,d)$ can be described with a degree $d^2$ linear recurrence relation.
\end{prop}

\begin{proof}
A partition $\lambda \in F(n,d)$ has $\lambda_d = d$, which means that its Durfee square is $d \times d$.  Therefore $\lambda$ can be decomposed into three parts: the Durfee square, a partition $\alpha$ below the Durfee square with first part at most $d$, and a partition $\beta'$ to the right of the Durfee square with first part at most $d - 1$ (since $\beta_1 = d$ would make $\lambda_d > d$; see Figure \ref{fig2}).  Since $1/(q;q)_k$ accounts for the partitions with parts at most $k$, the generating function follows.  The linear recurrence statement follows from the fact that $(q;q)_k$ is a polynomial of order $\binom{k+1}{2}$.
\end{proof}

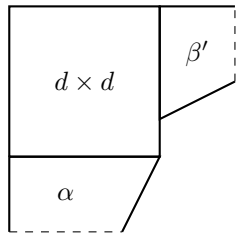
\begin{figure}[h]
\begin{tikzpicture}[scale=0.5]
\draw[thick] (0,0) -- (0,2) -- (4,2) -- (3,0);
\draw[thick] (0,2) -- (0,6) -- (4,6) -- (4,2) -- (0,2);
\draw[thick] (6,4) -- (4,3) -- (4,6) -- (6,6);
\draw[dashed] (6,4) -- (6,6);
\draw[dashed] (0,0) -- (3,0);
\node at (2,4) {$d \times d$};
\node at (1.5,1) {$\alpha$};
\node at (5,4.75) {$\beta'$};
\end{tikzpicture}
\caption{The Durfee square decomposition of a partition $\lambda$ with fixed point $\lambda_d = d$.} \label{fig2}
\end{figure}

The first few columns match sequences recorded in the Online Encyclopedia of Integer Sequences, namely
\[ f(n,2) = 2f(n-1,2) - 2f(n-3,2) + f(n-4,2)\]
with initial values $f(1,2) = f(2,2) = f(3,2) = 0$ and $f(4,2) = 1$ from the partition $(2,2) \vdash 4$ \cite[A002620]{o}, and
\begin{align*} f(n,3) & = 2f(n-1,3) + f(n-2,3) -3f(n-3,3) -f(n-4,3) \\
& \quad + f(n-5,3)  + 3f(n-6,3)  -f(n-7,3) -2f(n-8,3) + f(n-9,3)
\end{align*}
with initial values $f(1,3) = \cdots = f(8,3) = 0$ and $f(9,3) = 1$ from the partition $(3,3,3) \vdash 9$ \cite[A097701]{o}.

Examining an expanded version of the $f(n,d)$ triangle of Figure \ref{fig1} suggests that the columns `stabilize' to the sequence $1, 2, 5, 10, 20, \dots$, the convolution of $p(n)$ values $\sum_{i = 0}^n p(i) p(n-i)$ \cite[A000712]{o}.  More specifically, the first $d^2+d-1$ values of $f(n,d)$ are given by the following proposition. 

\begin{prop} \label{conv}
For a fixed $d \ge 1$, the sequence $f(n,d)$ satisfies $f(1,d) = \cdots = f(d^2-1,d) = 0$ and, for $0 \le k \le d-1$, 
\[f(n,d^2 + k) = \sum_{i = 0}^k p(i) p(k-i).\]
\end{prop}

\begin{proof}
First, note that $f(n,1) = 1$ for all $n \ge 1$ from the partition $1^n \vdash n$.  (Recall that, for a partition, an exponent denotes repetition.)

For $d \ge 2$, the smallest partition with a fixed point at $d$ is $d^d \vdash d^2$, thus $f(n,d) = 0$ for $n \le d^2 - 1$.

For $0 \le k \le d-1$, a partition in $F(d^2 + k, d)$ has Durfee square $d \times d$ and $k$ additional dots in its Ferrers diagram.  These $k$ dots can be split between the partitions below and to the right of the Durfee square, i.e., in the notation of Figure \ref{fig2}, $|\alpha| + |\beta| = k$.  Since $k < d$, these dots outside the Durfee square cannot affect the fixed point.  Therefore any pair of partitions $\alpha, \beta$ with $\alpha \vdash i$ and $\beta \vdash k - i$ for $0 \le i \le k$ gives a partition of $d^2 + k$ with fixed point $d$.
\end{proof}

One could extend this result using the convolution of $p(n)$ values by carefully counting exceptions (e.g., using the notation of Proposition \ref{conv} and its proof, the one case when $k=d$ that affects the fixed point is when $\alpha$ is the empty partition and $\beta = (d)$), but that approach soon becomes complicated.  

Instead, we improve the order $d^2$ linear recurrence for $f(n,d)$ provided in Proposition \ref{colrecur} by Theorem \ref{fixedPas} which expresses $f(n,d)$ in terms of four values above and left of it in the triangle, namely
\[f(n,d) = f(n-d+1,d) + f(n-d,d) - f(n-2d + 1, d) + f(n-2d+1,d-1).\]

\begin{proof}[Proof of Theorem \ref{fixedPas}]
We establish a bijection
\[F(n,d) \cup F(n-2d+1,d) \cong F(n-d+1,d) \cup F(n-d,d) \cup F(n-2d+1,d-1).\]

Given $\lambda \in F(n-2d+1,d-1)$, increase the first $d-2$ parts by one, add a part $d$, and increase the fixed point $\lambda_{d-1}$ by one.  That is, 
\[ (\lambda_1, \ldots, \lambda_{d-2}, d-1, \lambda_d, \ldots) \mapsto \mu = (\lambda_1+1, \ldots, \lambda_{d-2} + 1, d, d, \lambda_d, \ldots).\]
The image $\mu$ is a partition of $n-2d+1+(d-2)+d+1 = n$ with $\mu_d = d$, thus an element of $F(n,d)$.  Note also that $\mu_{d-1} = d$ and $\mu_{d+1} < d$. 

There are two cases for the image of a partition $\lambda \in F(n-d,d)$.  

First, if $\lambda_{d-1} = d$ (in addition to the fixed point $\lambda_d = d$), then add another part $d$ to make a partition $\mu \in F(n,d)$ with $\mu_{d-1} = \mu_d = \mu_{d+1} = d$.  

Second, if $\lambda_{d-1} > d$, then decrease the first $d-1$ parts by one to make a partition $\mu \in F(n-2d+1,d)$.

Given $\lambda \in F(n-d+1,d)$, increase the first $d-1$ parts by one to make a partition $\mu \in F(n,d)$.  Note that $\mu_{d-1} > d$.

As noted, the images of $F(n-2d+1,d-1)$, of $F(n-d+1,d)$, and from $F(n-d,d)$ into $F(n,d)$ are all distinct.

For the reverse map, given $\mu \in F(n-2d+1,d)$, increase the first $d-1$ parts by one to make a partition $\lambda \in F(n-d,d)$.  Note that $\lambda_{d-1} > d$.

There are three cases for the image of a partition $\mu \in F(n,d)$.  

First, if $\mu_{d-1} > d$, then decrease the first $d-1$ parts by one to make a partition $\lambda \in F(n-d+1,d)$.  

Second, if $\mu_{d-1} =  \mu_{d+1} = d \; (= \mu_d)$, then remove $\mu_{d+1}=d$ to make a partition $\lambda \in F(n-d,d)$. Note that $\lambda_{d-1} = d$.  

Third, if $\mu_{d-1} = d$ and $\mu_{d+1}<d$, then decrease the first $d-2$ parts by one, remove $\mu_{d-1} = d$, and decrease the fixed point $\mu_d = d$ by one, i.e.,
\[ (\mu_1, \ldots, \mu_{d-2}, d, d, \mu_{d+1}, \ldots) \mapsto \lambda = (\mu_1 -1, \ldots, \mu_{d-2}-1, d-1, \mu_{d+1}, \ldots).\]
The image $\lambda$ is a partition of $n-(d-2)-d-1 = n-2d+1$ with $\lambda_{d-1} = d-1$, thus an element of $F(n-2d+1,d-1)$.

As noted, the images of $F(n-2d+1,d)$ and from $F(n,d)$ into $F(n-d,d)$ are distinct.  

It is clear that the two maps are inverses, establishing the bijection.
\end{proof}

For example, Table \ref{ex0} shows the correspondence between $F(7,2) \cup F(4,2)$ and $F(6,2) \cup F(5,2) \cup F(4,1)$.

\begin{table}[h]
\renewcommand{\arraystretch}{1.2}
\begin{tabular}{r|ccccccc}
$F(7,2) \cup F(4,2)$ & 52 & 421 & 322 & 3211 & 2221 & 22111 & 22 \\ \hline
$F(6,2) \cup F(5,2) \cup F(4,1)$ & 42 & 321 & 222 & 2211 & 221 & 1111 & 32
\end{tabular}
\caption{Example of the bijection of proof of Theorem \ref{fixedPas}.} \label{ex0}
\end{table}

\section{Sums in the refined fixed point triangle}

We now prove the various sum results related to the $f(n,d)$ triangle.  Recall that $a(n)$ is the number of partitions $\lambda \vdash n$ such that the size of the Durfee square of $\lambda$ is not a part of $\lambda$.  First, we establish the generating function for $a(n)$ and show that the sequence counts a different family of partitions.

\begin{prop} \label{ab}
The sequence $a(n)$ satisfies
\[ \sum_{n \ge 1} a(n) \, q^n = \sum_{d \ge 1} \frac{q^{d^2+d}}{(q;q)_{d-1} (q;q)_d}.\]
Also, $a(n)$ counts the partitions $\mu$ of $n$ for which $\mu_d = \mu_{d+1} = d$, i.e., the $\mu \in F(n,d)$ with the additional constraint that $\mu_{d+1} = d$.
\end{prop}

\begin{proof}
The generating function argument is much like the proof of Proposition \ref{colrecur}.  Suppose $\lambda$ has Durfee square $d \times d$.  In order to not have $d$ as a part, it must be the case that $\lambda_d \ge d+1$ and $\lambda_{d+1} \le d-1$.  Referencing $\alpha$ and $\beta$ in the left-hand side of Figure \ref{fig3}, that means $\alpha_1 \le d-1$ and $\alpha$, thus the $1/(q;q)_{d-1}$ term.  Also, $\beta_1 = d$ which accounts for the additional $d$ in the exponent of $q$ in the numerator---this makes a $d \times (d+1)$ rectangle, called a 1-Durfee rectangle by the author, Sellers, and Yee \cite{hsy}.  The remaining dots to the right of the 1-Durfee rectangle have first part at most $d$, thus the $1/(q;q)_d$ term.

The other family of partitions counted by $a(n)$ follows by conjugation.  Given $\lambda$ counted by $a(n)$, let $\mu = \lambda'$.  Conjugation fixes the Durfee square, swaps $\alpha$ for $\beta$, and swaps $\beta'$ for $\alpha'$.  Since $\alpha_1 \le d-1$, now $\alpha'$ to the right of the Durfee squares leaves $\mu_d = d$.  Also, since $\beta_1 = d$, now $\beta$ below the Durfee square gives $\mu_{d+1} = d$.
\end{proof}

\begin{figure}[t]
\begin{tikzpicture}[scale=0.5]
\draw[thick] (0,0) -- (0,2) -- (3,2) -- (2,0);
\draw[thick] (0,2) -- (0,6) -- (4,6) -- (4,2) -- (0,2);
\draw[thick] (6,3) -- (5,2) -- (4,2) -- (4,6) -- (6,6);
\draw[dashed] (6,3) -- (6,6);
\draw[dashed] (0,0) -- (2,0);
\node at (2,4) {$d \times d$};
\node at (1.25,1) {$\alpha$};
\node at (5,4) {$\beta'$};
\end{tikzpicture}
\qquad \qquad
\begin{tikzpicture}[scale=0.5]
\draw[thick] (0,0) -- (0,2) -- (4,2) -- (4,1)--(3,0);
\draw[thick] (0,2) -- (0,6) -- (4,6) -- (4,2) -- (0,2);
\draw[thick] (6,4) -- (4,3) -- (4,6) -- (6,6);
\draw[dashed] (6,4) -- (6,6);
\draw[dashed] (0,0) -- (3,0);
\node at (2,4) {$d \times d$};
\node at (1.75,1) {$\beta$};
\node at (5,4.75) {$\alpha'$};
\end{tikzpicture}
\caption{On the left-hand side, a partition with a $d \times d$ Durfee square and no part $d$.  On the right-hand side, the conjugate of that partition.} \label{fig3}
\end{figure}
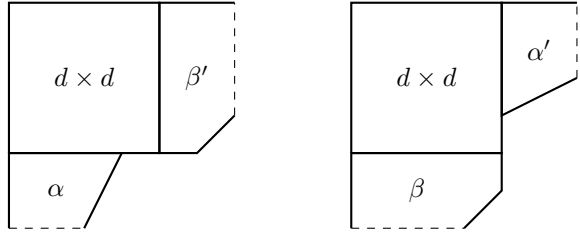

We can now prove the diagonal sum result.

\begin{proof}[Proof of Theorem \ref{diag}]
We establish a bijection
\[F(n,1) \cup F(n-1, 2) \cup F(n-2,3) \cup \cdots \cong B(n+1)\]
where $B(n+1)$ consists of the $\mu \in F(n+1,d)$ with $\mu_{d+1} = d$ also.  We know by Proposition \ref{ab} that these are counted by the sequence $a(n)$, as desired.

Given $d \ge 1$ and $\lambda \in F(n+1-d,d)$, let 
\[\mu = (\lambda_1, \ldots, \lambda_d-1, \lambda_d, d, \lambda_{d+1}, \ldots),\] 
i.e., insert an additional part $d$ as the new $(d+1)$st part of $\mu$.  The resulting $\mu$ is a partition of $n+1$ with $\mu_d = \mu_{d+1} = d$ as required.  

The reverse map is clear: Given $\mu \in B(n+1)$, remove the part $\mu_{d+1} = d$ to give a partition in $F(n+1-d,d)$.
\end{proof}

For example, Table \ref{ex1} shows the correspondence between $F(8,1) \cup F(7,2)$ and $B(9)$.  (Note that $3^3 \notin B(9)$ since the definition requires $\mu_3 = \mu_4 = 3$ while $(3,3,3)$ has only three parts.)

\begin{table}[h]
\renewcommand{\arraystretch}{1.2}
\begin{tabular}{r|ccccccc}
$F(8,1) \cup F(7,2)$ & $1^8$ & 52 & 421 & 322 & 3211 & 2221 & 22111 \\ \hline
$B(9)$ & $1^9$ & 522 & 4221 & 3222 & 32211 & 22221 & 222111
\end{tabular}
\caption{Example of the bijection of proof of Theorem \ref{diag}.} \label{ex1}
\end{table}

In Theorem \ref{row}, the row sums of the $f(n,d)$ triangle connect to the crank and mex statistics as follows.

\begin{proof}[Proof of Theorem \ref{row}]
Clearly the row sums $\sum_{k \ge 1} f(n,k)$ give the total number of partitions of $n$ that have a fixed point.  The author and Sellers have shown that this count matches the partitions of $n$ with even mex and also the number of partitions of $n$ with positive crank \cite[Theorem 2.1]{hs24}.
\end{proof}

For our last proof, recall that the $f(n,d)$ has finite antidiagonal sums as one can show that row $n$ has $\lfloor \sqrt{n} \rfloor$ nonzero entries (where $\lfloor \cdot \rfloor$ denotes the integer floor function).  The antidiagonal sums are the well known sequence $p(n)$.

\begin{proof}[Proof of Theorem \ref{antidiag}]
We establish a bijection
\[F(n,1) \cup F(n+1, 2) \cup F(n+2,3) \cup \cdots \cong P(n-1).\]

Given $d \ge 1$ and $\lambda \in F(n-1+d,d)$, let 
\[\mu = (\lambda_1, \ldots, \lambda_d-1, \lambda_{d+1}, \ldots),\] 
i.e., remove the fixed point.  The resulting $\mu$ is a partition of $n-1$.

For the reverse map, given $\mu \in P(n-1)$, find the greatest $d$ such that $\mu_{d-1} \ge d$ (this is a parameter for the 1-Durfee square of $\mu$, which can have size $0 \times 1$).
Let
\[\lambda = (\mu_1, \ldots, \mu_{d-1}, d, \mu_d, \ldots),\]
i.e., insert $\lambda_d = d$.  This $\lambda$ is a partition of $n-1+d$ with fixed point $d$, so $\lambda \in F(n-1+d,d)$.  

It is clear that these are inverse maps.
\end{proof}

For example, Table \ref{ex2} shows the correspondence between $F(7,1) \cup F(8,2) \cup F(9,3)$ and $P(6)$.

\begin{table}[h]
\renewcommand{\arraystretch}{1.2}
\begin{tabular}{r|ccccccccccc}
$F(7,1) \cup F(8,2) \cdots$ & $1^7$ & 62 & 521 & 422 & 4211 & 3221  \\ \hline
$P(6)$ & $1^6$ & 6 & 51 & 42 & 411 & 321  \\ \hline \hline
$\cdots F(8,2) \cup F(9,3)$ & 32111 & 2222 & 22211 & $221^6$ & 333 \\ \hline
$P(6)$ & 3111 & 222 & 2211 & 21111 & 33
\end{tabular}
\caption{Example of the bijection of proof of Theorem \ref{antidiag}.} \label{ex2}
\end{table}


\begin{thebibliography}{9}
\bibitem{ag} 
G.E.\ Andrews, F.G.\ Garvan. Dyson's crank of a partition.  \emph{Bull.\ Amer.\ Math.\ Soc.} 18 (1988) 167--171.  

\bibitem{an}
G.E.\ Andrews, D.\ Newman. The minimal excludant in integer partitions. \emph{J.\ Integer Seq.} 23 (2020) 20.2.3.

\bibitem{bk}
A.\ Blecher, A.\ Knopfmacher.  Fixed points and matching points in partitions.  \emph{Ramanujan J.} 58 (2022) 23--41.

\bibitem{d}
F.\ Dyson. Some guesses in the theory of partitions. \emph{Eureka} 8 (1944) 10--15. 

\bibitem{hs20}
B.\ Hopkins, J.A.\ Sellers. Turning the partition crank. \emph{Amer.\ Math.\ Monthly} 127 (2020) 654--657.

\bibitem{hs24}
B.\ Hopkins, J.A.\ Sellers.  On Blecher and Knopfmacher's fixed points for integer partitions.  Submitted.  \url{https://arxiv.org/abs/2305.05096}, 2023.

\bibitem{hsy}
B.\ Hopkins, J.A.\ Sellers, A.J.\ Yee.  Combinatorial perspectives on the crank and mex partition statistics.  \emph{Electron.\ J.\ Combin.} 29 (2022) P2.9.

\bibitem{o}
N.J.A.\ Sloane, ed. The Online Encyclopedia of Integer Sequences. \url{https://oeis.org}, 2023.

\end{thebibliography}
\end{document}